\documentclass[11pt]{amsart}
\usepackage[margin=1in]{geometry}

\usepackage{MnSymbol}
\usepackage{mathrsfs}
\usepackage{colonequals}

\usepackage{hyperref}
\hypersetup{
pdflang={en-US},
pdftitle={Stability of Tschirnhausen Bundles},
pdfauthor={Izzet Coskun, Eric Larson, and Isabel Vogt},
pdfsubject={Algebraic Geometry},
pdfkeywords={Hurwitz spaces, Tschirnhausen bundles, stability}
}

\usepackage{tikz-cd}

\newcommand{\pp}{\mathbb{P}}
\newcommand{\zz}{\mathbb{Z}}

\renewcommand{\O}{\mathcal{O}}

\newcommand{\sH}{\mathscr{H}}

\newcommand{\Pic}{\operatorname{Pic}}
\newcommand{\rk}{\operatorname{rk}}
\newcommand{\Tsch}{\operatorname{Tsch}}
\newcommand{\DTs}{\operatorname{Tsch^\vee}}
\newcommand{\et}{\text{\'et}}
\newcommand{\pr}{\text{pr}}
\newcommand{\Gal}{\operatorname{Gal}}

\newcommand{\PGL}{\pp\mkern-2mu\operatorname{GL}}
\newcommand{\PSL}{\pp\mkern-2mu\operatorname{SL}}
\newcommand{\Sym}{\operatorname{Sym}}
\newcommand{\End}{\operatorname{End}}

\newcommand{\negmod}{\rightarrow}

\newcommand{\defi}[1]{\textsf{#1}} 

\newtheorem{thm}{Theorem}[section]
\newtheorem{lem}[thm]{Lemma}
\newtheorem{prop}[thm]{Proposition}
\newtheorem{cor}[thm]{Corollary}

\theoremstyle{definition}
\newtheorem{rem}[thm]{Remark}
\newtheorem{example}[thm]{Example}
\newtheorem{ques}[thm]{Question}

\title{Stability of Tschirnhausen Bundles}
\author{Izzet Coskun}
\address{Department of Mathematics, Statistics, and CS \\
University of Illinois at Chicago, Chicago IL 60607}
\email{icoskun@uic.edu}
\author{Eric Larson}
\address{Department of Mathematics, Brown University}
\email{elarson3@gmail.com}
\author{Isabel Vogt}
\address{Department of Mathematics, Brown University}
\email{ivogt.math@gmail.com}

\thanks{During the preparation of this article, I.C.\ was supported
by NSF FRG grant DMS-1664296 and NSF grant DMS-2200684,  E.L. was supported by
NSF  grants DMS-1802908 and DMS-2200641, and I.V. was supported by NSF grants DMS-1902743 and DMS-2200655.}

\begin{document}

\begin{abstract}
Let \(\alpha\colon X \to Y\) be a general degree \(r\) primitive map of nonsingular, irreducible, projective curves over an algebraically closed field of characteristic zero or larger than \(r\). We prove that the Tschirnhausen bundle of \(\alpha\) is semistable if \(g(Y) \geq 1\) and stable if \(g(Y) \geq 2\).    
    
\end{abstract}

\maketitle

\section{Introduction}

Let \(\alpha\colon X \to Y\) be a finite map of degree \(r\) between nonsingular, irreducible, projective curves of genera \(g(X)=g\) and \(g(Y)=h\), defined over an algebraically closed field of characteristic zero or larger than \(r\). The natural inclusion \(\O_Y \to \alpha_* \O_X\) given by pullback of functions admits a splitting by \(1/r\) times the trace map.  Hence, \[\alpha_* \O_X = \O_Y \oplus \DTs(\alpha)\] for a vector bundle \(\DTs(\alpha)\). The dual of \(\DTs(\alpha)\) is called the \defi{Tschirnhausen bundle} \(\Tsch(\alpha)\) associated to \(\alpha\).

The relative canonical embedding maps \(X\) into \(\mathbb{P} (\Tsch(\alpha))\).  Consequently, \(\Tsch(\alpha)\) plays a central role in the study of covers of curves and the geometry of the Hurwitz scheme.
A natural question raised, for example in \cite{dp22} and \cite{ll22} (and similar to questions raised in \cite{beauville}), is the following:

\begin{ques}\label{ques:main}
When is \(\Tsch(\alpha)\) (semi)stable?
\end{ques}

Recall that the \defi{slope} \(\mu(V)\) of a vector bundle \(V\) on a curve is  the ratio
\(\mu(V) = \frac{\deg(V)}{\rk(V)}\). The bundle \(V\) is called \defi{semistable} if, for every proper subbundle \(W\), we have
\(\mu(W) \leq \mu(V)\). The bundle \(V\) is \defi{stable} if the inequality is strict for every such \(W\). Semistable bundles are the building blocks of all bundles on a curve, via the Harder--Narasimhan filtration, and hence play a central role in their theory.  The moduli space of semistable vector bundles of fixed rank and degree is an irreducible projective variety.  A positive answer to Question~\ref{ques:main} implies the existence of a rational map from the Hurwitz scheme to the space of semistable vector bundles, and raises the possibility of understanding the moduli of covers via the moduli of bundles.

 The map \(\alpha\colon X \to Y\) is said to be \defi{primitive} if the induced map on \'etale fundamental groups \(\alpha_*\colon\pi_1^{\et}(X)\to \pi_1^{\et}(Y)\) is surjective.  In general, every finite map \(\alpha\colon X \to Y\) factors uniquely into a primitive map \(\alpha^\pr \colon X \to Z\) followed by an \'etale map \(\alpha^\et \colon Z \to Y\), where \(Z\) corresponds to the image \(\alpha_* \pi_1^\et(X)\) in  \(\pi_1^{\et}(Y)\). Since \(\DTs(\alpha^\et)\) is naturally a direct summand of \(\DTs(\alpha)\), while \(\deg(\DTs(\alpha^\et))=0\) and \(\deg(\DTs(\alpha))<0\) when \(\alpha\) is not \'etale (see Lemma \ref{lem-degTsch}), we conclude the following proposition.

\begin{prop}\label{prop-unstable}
Suppose the finite map \(\alpha\colon X \to Y\) has a nontrivial factorization \(\alpha = \alpha^\et \circ \alpha^\pr\), where \(\alpha^\pr\colon X \to Z\) is primitive and \(\alpha^\et\colon Z \to Y\) is \'etale. Then \(\Tsch(\alpha)\) is unstable.
\end{prop}
The (semi)stability problem for Tschirnhausen bundles is therefore only interesting when \(\alpha\) is either primitive or \'etale.

\subsection*{The \'etale case} If \(\alpha \colon X \to Y\) is \'etale, then the Galois group \(\Gal(X/Y)\) acts on the fibers of \(\alpha\) by permutation. Fixing a basepoint, this action realizes \(\Gal(X/Y)\) as a subgroup of the symmetric group \(\mathfrak{S}_r\). We will study this case in Section~\ref{sec-etale}.  The following proposition, which over the complex numbers is a simple case of the Narasimhan--Seshadri correspondence \cite{ns65},  characterizes the stability of Tschirnhausen bundles in the \'etale case.

\begin{prop}\label{prop-etale}
Let \(\alpha\colon X \to Y\) be an \'etale cover of degree \(r\). Then \(\Tsch(\alpha)\) is semistable. Furthermore, \(\Tsch(\alpha)\) is stable if and only if the restriction of the standard representation of \(\mathfrak{S}_r\) to \(\Gal(X/Y)\) is irreducible.
\end{prop}

\subsection*{The primitive case} When the characteristic of the base field is zero or greater than \(r\), then the Hurwitz space of primitive degree \(r\) genus \(g\) covers of \(Y\) is irreducible if the degree of the ramification divisor is positive (see \S \ref{subsec-irreducible}) and otherwise empty. Hence, we can refer to the general degree \(r\) genus \(g\) cover of \(Y\). The main theorem of our paper, whose proof is in Sections~\ref{sec-genus1} and \ref{sec-generalcase}, is the following.

\begin{thm}\label{thm:tsch}
Let \(\alpha \colon X \to Y\) be a general primitive degree \(r\) cover, where \(X\) has genus \(g\) and \(Y\) has genus \(h\), over an algebraically closed field of characteristic zero or greater than \(r\).  
\begin{enumerate}
   \item\label{main_h1} If \(h=1\), then \(\Tsch(\alpha)\) is  semistable.
    \item\label{main_h2} If \(h\geq 2\), then \(\Tsch(\alpha)\) is stable.
\end{enumerate}
\end{thm}

\begin{rem}
When the characteristic of the base field is greater than \(r\), one can strengthen (semi)stable in Theorem \ref{thm:tsch} to strongly (semi)stable (see Corollary \ref{cor-strong}).
\end{rem}

\begin{rem}
By the Atiyah classification \cite{at57},  on a curve of genus 1 there exist stable bundles if and only if the degree and the rank are coprime. Hence, in general semistability is the best one can hope for.
\end{rem}

\begin{rem}
By the Birkhoff--Grothendieck Theorem, every vector bundle on \(\pp^1\) splits as a direct sum of line bundles \(V \simeq \bigoplus_{i=1}^{\rk(V)} \O_{\pp^1}(e_i)\). Hence on \(\pp^1\), there are no stable bundles of rank larger than one, and there exist semistable bundles only if the rank divides the degree.
The best one can hope for is thus that the bundle is \defi{balanced}, i.e., \(|e_i -e_j| \leq 1\) for all \(i\) and \(j\).
When the rank divides the degree, a balanced vector bundle on \(\pp^1\) is called \defi{perfectly balanced} and is semistable.

\begin{prop}[\cite{balico89}]\label{thm-genus0}
Let \(\alpha \colon X \to \pp^1\) be a general degree \(r\) cover, where \(X\) has genus \(g\), defined over an algebraically closed field of characteristic zero or greater than \(r\). Then \(\Tsch(\alpha)\) is balanced.
\end{prop}

 For completeness, we will give a short proof of this well-known fact in \S \ref{sec-rational}. Considering the loci in the Hurwitz space where the Tschirnhausen bundle is \textit{unbalanced} yields an important stratification.  For example,
when the degree of \(\alpha\) is \(3\), we have \(\Tsch(\alpha) \simeq \O_{\pp^1}(e_1) \oplus \O_{\pp^1}(e_2)\) and the difference \(e_1- e_2\) is the classical \defi{Maroni invariant} of the trigonal curve.
In general, the possible splitting types are known as \defi{scrollar invariants}.  Classifying scrollar invariants has attracted a lot of attention (see, for example, \cite{balico, cvz22, dp22, Vem22}).
\end{rem}

\begin{rem}
There are fascinating number-theoretic analogues of Tschirnhausen bundles.
For example, the integers appearing in the splitting type of the Tschirnhausen bundle on \(\pp^1\) have close connections to the so-called ``shapes of number fields'', see \cite{Vem22}.
\end{rem}

\subsection*{Strategy} In \S \ref{sec-etale}, we first study the case of \'etale covers and prove Proposition \ref{prop-etale}. We then analyze the case \(h=1\), by specializing \(X\) to a nodal curve to leverage our understanding of the \'etale case. We finally use the theory of admissible covers to specialize the base \(Y\) to a reducible curve. This allows us to prove Theorem \ref{thm:tsch} by induction on \(h\).

\subsection*{History and related work} Kanev proves the stability of \(\Tsch(\alpha)\) for general coverings of degree \(r \leq 5\) for sufficiently large \(g\) in his work on the unirationality of Hurwitz spaces \cite{kanev}.

Deopurkar and Patel in \cite[Theorem 1.5]{dp22} prove that \(\Tsch(\alpha)\) is stable for general coverings when \(h \geq 2\) and \(g\) is sufficiently large (depending on \(h\) and \(r\)). In fact, they prove that the rational map from the Hurwitz space to the moduli space of stable bundles on \(Y\) is dominant if \(g \gg 0\). More generally, they asymptotically answer the question of when a bundle can be realized as the Tschirnhausen bundle associated to a cover. By contrast, our Theorem \ref{thm:tsch} applies even in cases when the map from the Hurwitz space to the moduli space of stable bundles cannot be dominant.

Landesman and Litt \cite[Example 1.3.8]{ll22} prove that the Tschirnhausen bundle is semistable under suitable assumptions and a better bound on the genus \(g\). More generally, their methods apply to the semistability of the vector bundle associated to any irreducible \(G\)-representation \(\rho\) and general \(G\)-cover \(\alpha \colon X \to Y\), when \(\dim \rho < 2 \sqrt{g + 1}\). This semistability problem for \(G\)-covers with \(g\) small, and the stability for any \(g\), remain open (except for the standard representation of \(\mathfrak{S}_r\) which is the content of Theorem~\ref{thm:tsch}).

\subsection*{Acknowledgements} We would like to thank Anand Deopurkar, Aaron Landesman, Daniel Litt and  Anand Patel for rekindling our interest in the problem and for invaluable discussions. The first author would like to thank Chien-Hao Liu and Prof. Shing-Tung Yau for bringing the problem to his attention and many valuable conversations. Finally, we thank the anonymous referees for helpful feedback that have improved the paper. 

\section{Preliminaries}
\noindent
In this section, we collect basic facts concerning Tschirnhausen bundles and Hurwitz schemes.

\subsection{Tschirnhausen bundles} Let \(\alpha\colon X \to Y\) be a finite map of degree \(r\) between nonsingular, irreducible, projective curves of genera \(g(X)=g\) and \(g(Y)=h\). Let \(\DTs(\alpha)\) be the dual of the  Tschirnhausen bundle associated to \(\alpha\).

\begin{lem}\label{lem-degTsch}
Let \(b\) be the degree of the ramification divisor of \(\alpha\colon X \to Y\). Then \(\DTs(\alpha)\) is a vector bundle of rank \(r-1\) and degree \(-b/2\). In particular, if \(\alpha\) is \'etale, then \(\deg(\DTs(\alpha))=0\), and if \(\alpha\) is branched, then \(\deg(\DTs(\alpha)) < 0\).
\end{lem}

\begin{proof}
Since \(\alpha_* \O_X\) has rank \(r\), the rank of \(\DTs(\alpha)\) is \(r-1\). The relation  \(\alpha_* \O_X = \O_Y \oplus \DTs(\alpha)\) implies that \(\deg(\alpha_* \O_X)= \deg(\DTs(\alpha))\). Hence, by the Riemann--Roch formula \[1-g = \chi(\O_X)=\chi(\alpha_* \O_X)= \deg(\DTs(\alpha)) + r(1-h).\] By the Riemann--Hurwitz formula, we have
\(g-1= r(h-1) + \frac{1}{2} b\), where \(b\) is the degree of the ramification divisor. We conclude that
\[\deg (\DTs(\alpha)) = - \frac{b}{2}. \qedhere\]
\end{proof}

 Recall that when the degree and the rank are coprime, a semistable bundle is automatically stable. Similarly, a bundle is (semi)stable if and only if its dual is (semi)stable. In positive characteristic, a bundle \(V\) is \defi{strongly (semi)stable} if \(V\) is (semi)stable and remains (semi)stable under pullback by all iterates of the Frobenius morphism.

\subsection{Irreducibility of the primitive component in large characteristic}\label{subsec-irreducible}

Fix a curve \(Y\) of genus \(h \geq 1\), and let \(\sH_{r,g}(Y)\) denote the simply-branched
Hurwitz space, parameterizing primitive degree \(r\) genus \(g\) covers of \(Y\)
with a positive number of branch points.
When the ground field is the complex numbers, topological calculations of Gabai and Kazez \cite{gk90}
(generalizing earlier calculations of Clebsch \cite{clebsch} in the case \(Y = \pp^1\))
establish that \(\sH_{r,g}(Y)\) is non-empty and irreducible.
Fulton \cite{fulton} showed that \(\sH_{r,g}(\pp^1)\) is irreducible in any characteristic larger than \(r\) by reduction to characteristic zero and application of Clebsch's result.
However, Fulton's reduction argument does not rely essentially on the base being \(\pp^1\). Using Fulton's reduction combined with Gabai and Kazez's result therefore yields the following.

\begin{prop}[Clebsch, Fulton, Gabai--Kazez]\label{prop-Hurwitz}
Let \(Y\) be a smooth and irreducible curve defined over an algebraically closed field of characteristic larger than \(r\).  Then the primitive, simply-branched Hurwitz space \(\mathscr{H}_{r,g}(Y)\) of degree \(r\), genus \(g\) covers of \(Y\) is irreducible.
\end{prop}

\begin{proof}
We apply the argument in \cite{fulton}, mutatis mutandis, to reduce to characteristic zero
(in which case the proposition holds by \cite{gk90}).
For the reader's convenience, we briefly sketch the reduction argument of \cite{fulton}.

Extracting the branch locus gives a natural \'etale map
\(\sH_{r,g}(Y) \to \operatorname{Sym}^{2g - 2 - r(2h - 2)}(Y) \smallsetminus \Delta\).
If in addition the characteristic exceeds \(r\), then by using Abhyankar's lemma \cite[Lemma 3.6 of Expose X]{SGA1},
one can show that this morphism is proper, and thus finite \'etale.
(For details, see \cite[Theorem 7.2]{fulton}, which is stated in the case \(Y = \pp^1\),
but the argument given works in general.)

We claim that the diagonal \(\Delta \subset \operatorname{Sym}^{2g - 2 - r(2h - 2)}(Y)\) remains reduced
in odd positive characteristic.  When \(Y = \pp^1\), this follows from \cite[Lemma A.3]{fulton}.
The general case follows by a simple modification of the proof given in 
\cite[Lemma A.3]{fulton}, 
as we now explain.
Set \(w = 2g - 2 - r(2h - 2)\).
Consider an arbitrary element \(Q\) of \( \operatorname{Sym}^w(Y)\),
corresponding to a multiset of points with multiplicities
\(w_1 + w_2 + \cdots + w_k = w\).
We want to show that \(\Delta\) is reduced in a neighborhood of \(Q\).
In an \'etale neighborhood of \(Q\), the scheme \(\operatorname{Sym}^w(Y)\) is isomorphic
to \(\operatorname{Sym}^{w_1}(Y) \times \operatorname{Sym}^{w_2}(Y) \times \cdots \times \operatorname{Sym}^{w_k}(Y)\),
under which the diagonal is the union of pullbacks of the diagonals from each factor.
It therefore suffices to establish the claim when \(w = w_1\), i.e., all of the points in the multiset \(Q\) coincide.
Choose a local coordinate \(y\) for \(Y\) at this common point,
so that \(b_1, b_2, \ldots, b_w\) give local coordinates for \(\operatorname{Sym}^w(Y)\) in a neighborhood of \(Q\)
by associating to a polynomial \(P = y^w + b_1 y^{w-1} + \cdots + b_w\) its roots \(\xi_1, \xi_2, \ldots, \xi_w\).
The equation of \(\Delta\) is then the discriminant of \(P\).
Suppose for sake of contradiction that this discriminant factors as \(fg^2\) with \(g\) nonconstant.
Substituting the \(b_i\) for the elementary symmetric functions of the \(\xi_j\),
the polynomial \(g\) is divisible by some \(\xi_i - \xi_j\).
Since \(g\) is symmetric, it is thus divisible by \(\prod_{i<j} (\xi_i - \xi_j)\).
Up to rescaling, we must therefore have \(g = \prod_{i<j} (\xi_i - \xi_j)\) and \(f=1\).
But in odd characteristic, \(\prod_{i<j} (\xi_i - \xi_j)\) is not a symmetric function,
since under exchanging \(\xi_1\) and \(\xi_2\), it is sent to its negative.
This establishes that \(\Delta\) is reduced as claimed.

Since an irreducible finite \'etale cover of a smooth projective variety
cannot specialize to a reducible one unless the branch locus becomes nonreduced
(see \cite[Theorem 3.3]{fulton}), we conclude from the irreducibility of \(\sH_{r,g}(Y)\) in characteristic zero that
\(\sH_{r,g}(Y)\) remains irreducible in characteristic larger than \(r\).
\end{proof}

\section{The \'etale case}\label{sec-etale}

In this section,  we study the case when \(\alpha\colon X \to Y\) is \'etale and prove Proposition \ref{prop-etale}. 
Let \(\alpha\colon X \to Y\) be a degree \(r\) \'etale cover of nonsingular, irreducible, projective algebraic curves. Since the characteristic of the base field is zero or larger than \(r\), the map \(\alpha\) is separable. Let \(\beta\colon Z \to Y\) denote the  Galois closure of \(\alpha\) with Galois group \(G = \Gal(X/Y)\).  The \'etale cover \(\alpha\colon X \to Y\) corresponds to a subgroup \(H \subset G\).

\begin{proof}[Proof of Proposition \ref{prop-etale}]
We have that  \(\beta^* (\alpha_* \O_Y) = \O_Z^{\oplus r}\). Suppose that \(F \subset \DTs(\alpha)\) is a subbundle. Then \(\beta^* F \subset \beta^* \DTs(\alpha) \subset \O_Z^{\oplus r}\). We conclude that \(\deg(\beta^* F) \leq 0\). Hence, \(\deg(F) \leq 0\). Since \(\deg(\DTs(\alpha))=0\), it follows that \(\DTs(\alpha)\) is semistable.

To determine when \(\DTs(\alpha)\) is stable,
we start by choosing coset representatives \(\{ g_i H \}_{i=1}^r\) of \(H\) in \(G\). Then the bundle 
 \[ V= \bigoplus_{g_i H} \O_Z\] has a \(G\)-action acting by permutations on the cosets,
which is \(\beta\)-equivariant with respect to the natural action of \(G\) on \(Z\).
Let \(P\) be the permutation representation of the symmetric group \(\mathfrak{S}_r\) on an \(r\)-dimensional vector space. The representation \(P\) decomposes as a direct sum of two irreducible representations, the trivial representation \(T\) of dimension one and the standard representation \(S\) of dimension \(r-1\). This induces a decomposition
\[V= \bigoplus_{g_i H} \O_Z = (S \otimes \O_Z) \oplus (T \otimes \O_Z),\]
and a \(\beta\)-equivariant \(G\)-action on the bundle \(E\colon= S \otimes \O_Z\). Since the quotient of \(V\) by \(G\) is the bundle
\(\alpha_* \O_X\) on \(Z/G = X\), the quotient of \(E\) by \(G\) is \(\DTs(\alpha)\).  Suppose that \(F \subset \DTs(\alpha)\) is a proper subbundle of \(\DTs(\alpha)\) of rank \(k\) such that \(\mu(F) =0\), so that \(\DTs(\alpha)\) is strictly semistable. Then \(\beta^* F \subset E\) has slope \(0\). Since \(E\) is trivial, we conclude that \(\beta^*F \simeq \O_Z^{\oplus k} \). 

For any \(g \in G\), consider the action of \(g\) on \(\beta^* F\). This corresponds to a \(k \times k\) matrix of regular functions on \(Z\). Since \(Z\) is proper, this matrix must be a constant matrix. Hence, the action does not depend on the choice of point. We conclude that \(\beta^* F = V \otimes \O_Z\) for a subrepresentation \(V \subset S\). Hence \(S\) is a reducible representation.

Conversely, suppose \(S\) is a reducible representation of \(G\) with a subrepresentation \(W\). Then the quotient of \(W \otimes \O_Z\) by \(G\) is a subbundle of \(\DTs(\alpha)\) of slope zero and rank the dimension of \(W\). We conclude that \(\DTs(\alpha)\) is strictly semistable.
\end{proof}

\begin{rem}\label{rem-subrep}
The proof of Proposition \ref{prop-etale} in fact shows that any degree 0 subbundle of \(\DTs(\alpha)\) corresponds to a subrepresentation of the restriction of the standard representation.
\end{rem}

\begin{example}\label{ex:cyclic}
When \(\alpha\) is a cyclic Galois cover of degree \(r >2\), the standard representation is never irreducible, and consequently \(\Tsch(\alpha)\) is never stable.  Explicitly, if \(\alpha\) corresponds to an \(r\)-torsion line bundle \(\vartheta\) on \(Y\), then \(\Tsch(\alpha) \simeq \vartheta \oplus \cdots \oplus \vartheta^{r-1}\) \cite[Remark 4.1.7]{lazarsfeld1}.
\end{example}

\begin{example}
One can construct some interesting examples where the restriction of the standard representation is irreducible. Let \(G \subset \mathfrak{S}_r\) be a subgroup of the symmetric group. Denote the order of \(G\) by \(|G|\). For an element \(g \in G\), let \(F(g)\) denote the number of fixed points of \(g\) in the permutation representation of \(\mathfrak{S}_r\). By character theory, the restriction of the standard representation to \(G\) is irreducible if and only if  
\begin{equation}\label{eqn-character}
    \sum_{g \in G} F(g)^2 = 2 |G|.
\end{equation}
For \(r> 3\), the alternating group \(A_r\) satisfies \eqref{eqn-character}. To see this, use the identity $$\sum_{g\in G} F(g)^2 = \sum_{1 \leq i, j \leq r} \#\{g \in G : g \  \mbox{fixes} \  i, j\}$$ for $G = A_r$ and separately consider the cases $i=j$ and $i \not=j$ to obtain
$$\sum_{1 \leq i\not=j \leq r} \#\{g \in A_r : g \  \mbox{fixes} \  i, j\} + \sum_{1 \leq i \leq r} \#\{g \in A_r : g \  \mbox{fixes} \  i\} = \binom{r}{2} \cdot 2 \frac{(r-2)!}{2} + r \cdot \frac{(r-1)!}{2} = 2 | A_r|.$$  

Let \(q\) be a prime power relatively prime to \(n\).  The permutation action of \(\PGL_2(\mathbb{F}_q)\) or \(\PSL_2(\mathbb{F}_q)\) acting on \(\pp^1\) realizes these groups as subgroups of \(\mathfrak{S}_r\), where \(r = q+1\). By separately considering the different Jordan canonical forms, one can check that these groups satisfy \eqref{eqn-character}.  For example in the case of $\PGL_2(\mathbb{F}_q)$, we have $(q-2) \binom{q+1}{2}$ elements with exactly 2 fixed points, $(q-1)(q+1)$ elements with exactly 1 fixed point and the identity has $q+1$ fixed points. This immediately yields \eqref{eqn-character}. Similarly, the group of affine transformations on \(\mathbb{A}^1(\mathbb{F}_q)\) is naturally a subgroup of \(\mathfrak{S}_q\) and by an analogous  case work satisfies \eqref{eqn-character}.
\end{example}

\section{Covers of an elliptic curve}\label{sec-genus1}

We will construct genus \(g\) covers of an elliptic curve \(E\) with semistable Tschirnhausen bundle by  identifying \(g\) pairs of points in fibers of a cyclic \'etale cover \(\alpha \colon E' \to E\).  The following lemma describes how identifying two points in a fiber affects the resulting Tschirnhausen bundle.

Let \(\alpha \colon X \to Y\) be a degree \(r\) cover and let \(p \in Y\) be a point.  Then the fiber \(\alpha_*\O_X|_p\) is identified with functions on the fiber \(\alpha^{-1}(p)\).  For any two points \(q, q' \in \alpha^{-1}(p)\), there is a codimension \(1\) subspace \(\Delta_{q, q'} \subset \alpha_*\O_X|_p\) where the values at \(q\) and \(q'\) agree.   

Given a vector bundle \(V\) on a curve \(Y\), a Cartier divisor \(D \subset Y\) and a subspace \(\Delta \subseteq V|_D\), the \defi{negative elementary modification of \(V\) along \(D\) towards \(\Delta\)}, denoted \(V[D \negmod \Delta]\), is defined by the exact sequence
\[ 0 \to V[D \negmod \Delta] \to V \to V|_D/\Delta \to 0.\]

Let \(X'\) be the nodal curve obtained from \(X\) by identifying two points \(q\) and \(q'\) in the fiber over \(p \in Y\) and write \(\alpha' \colon X' \to Y\) for the new cover.

\begin{lem}\label{lem:Tsch_mod}
The bundle \(\DTs(\alpha')\) is the negative elementary modification \(\DTs(\alpha)[p \to \Delta_{q, q'}]\) of \(\DTs(\alpha)\) towards \(\Delta_{q,q'}\), which naturally fits in the exact sequence
 \[ 0 \to \DTs(\alpha') \to \DTs(\alpha) \to \DTs(\alpha)|_p / \Delta_{q_1,q_2} \to 0.\]
\end{lem}
\begin{proof} This follows from the fact that \(\alpha'_*\O_{X'}\) is a negative elementary modification of \(\alpha_*\O_X\) towards \(\Delta_{q, q'}\) and the following commutative diagram.
\begin{center}
\begin{tikzcd}
 & & 0 & 0 \\
&& \O_{\alpha^{-1}(p)}/\Delta_{q, q'} \arrow[r] \arrow[u]& \DTs(\alpha)|_p/\Delta_{q, q'} \arrow[u] \\
0 \arrow[r] & \O_Y \arrow[r] \arrow[d, equals] & \alpha_*\O_X  \arrow[r] \arrow[u]& \DTs(\alpha) \arrow[r] \arrow[u]& 0\\
 0 \arrow[r] & \O_Y \arrow[r] & \alpha'_*\O_{X'} \arrow[u] \arrow[r] & \DTs(\alpha') \arrow[r] \arrow[u]& 0 \\
 & & 0 \arrow[u] & 0 \arrow[u] 
\end{tikzcd}
\end{center}
\end{proof}

Using Lemma \ref{lem:Tsch_mod}, we will produce a genus \(g\) degree \(r\) cover of \(E\) with semistable Tschirnhausen bundle.  Let \(\alpha \colon E' \to E\) be a cyclic \'etale cover of \(E\) of degree \(r\).  The action of a choice of generator \(\tau\) for the cyclic subgroup \(\zz/r\zz \simeq \ker \alpha_* \subseteq \Pic^0E'[r]\) yields an automorphism \(\tau \colon E' \to E'\) that commutes with the map \(\alpha\), i.e., $\tau$ induces a cyclic ordering of the fibers of \(\alpha\).  Let \(q_1, \dots, q_g \in E'\) be \(g\) general points in \(E'\) and let \(q_i' \colonequals \tau(q_i)\).  Write \(p_i = \alpha(q_i) = \alpha(q_i')\).  This is illustrated in the diagram below.
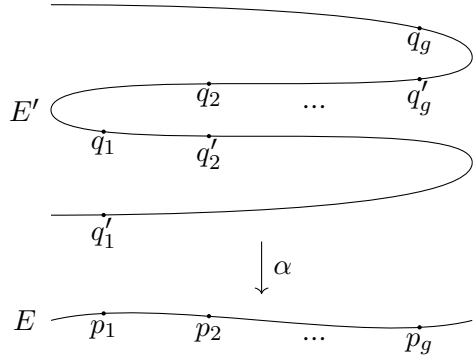
\begin{center}
\begin{tikzpicture}[scale=.7]
\draw (0, 0) .. controls (2, 0.5) and (6, -0.5) .. (8, 0);
\draw[->] (4, 1.5) -- (4, 0.5);
\draw (4.4, 1) node{\(\alpha\)};
\draw (0, 2) .. controls (1, 2) and (8, 2) .. (8, 3);
\draw (0, 4) .. controls (0, 3) and (8, 4) .. (8, 3);
\draw (0, 4) .. controls (0, 5) and (8, 4) .. (8, 5);
\draw (0, 6) .. controls (1, 6) and (8, 6) .. (8, 5);
\draw (-0.5, 0) node{\(E\)};
\draw (-0.5, 4) node{\(E'\)};
\filldraw (1, 0.132) circle[radius=0.03];
\filldraw (3, 0.08) circle[radius=0.03];
\filldraw (7, -0.132) circle[radius=0.03];
\draw (1, -0.17) node{\(p_1\)};
\draw (3, -0.23) node{\(p_2\)};
\draw (5, -0.40) node{\(\cdots\)};
\draw (7, -0.44) node{\(p_g\)};
\draw (5, 4) node{\(\cdots\)};
\filldraw (7, 4.585) circle[radius=0.03];
\filldraw (7, 5.553) circle[radius=0.03];
\filldraw (3, 4.498) circle[radius=0.03];
\filldraw (3, 3.502) circle[radius=0.03];
\filldraw (1, 2.003) circle[radius=0.03];
\filldraw (1, 3.587) circle[radius=0.03];
\draw (1, 3.32) node{\(q_1\)};
\draw (1, 1.67) node{\(q_1'\)};
\draw (3, 4.23) node{\(q_2\)};
\draw (3, 3.17) node{\(q_2'\)};
\draw (7, 4.23) node{\(q_g'\)};
\draw (7, 5.27) node{\(q_g\)};
\end{tikzpicture}
\end{center}
\noindent
Write \(X\) for the curve of genus \(g\) obtained from \(\alpha \colon E' \to E\) by identifying  \(q_i\) and \(q_i'\) for $1 \leq i \leq g$.  Write \(\alpha' \colon X \to E\) for the induced (ramified) degree \(r\) cover of \(E\).

\begin{prop}\label{prop-coverelliptic}
With the setup above, the Tschirnhausen bundle of the resulting degree \(r\) cover \(\alpha' \colon X \to E\) is semistable.
\end{prop}
\begin{proof}
By Lemma \ref{lem:Tsch_mod}, we have
\[\DTs(\alpha') \simeq \DTs(\alpha)[p_1 \to \Delta_{q_1, q_1'}] \cdots [p_g \to \Delta_{q_g, q_g'}]. \]
The bundle \(\DTs(\alpha')\) has degree \(-g\).

Suppose to the contrary that \(\DTs(\alpha')\) was not semistable, and let \(F \subset \DTs(\alpha')\) be the maximal destabilizing subbundle, which is also a subsheaf of \(\DTs(\alpha)\).  Since \(\DTs(\alpha)\) is semistable by Proposition \ref{prop-etale},
and of slope \(0\) by Lemma~\ref{lem-degTsch}, the slope of \(F\) is at most \(0\). 
Furthermore, by Remark \ref{rem-subrep}, any degree \(0\) subsheaf of \(\DTs(\alpha)\) corresponds to a subrepresentation of the standard representation of \(\Gal(E'/E) = \zz/r\zz\). Since the only subrepresentation of the standard representation of \(\zz/r\zz\) contained in a diagonal (where two coordinates agree) is zero, we conclude that no such degree \(0\) subbundle could be contained in \(\DTs(\alpha')\).
We therefore have
\[-g < \left( \frac{r-1}{\rk F}\right) \deg (F) < 0.\]
We will conclude the proof by showing that \(\deg(F)\) is divisible by \(g\).  Since \(\rk F \leq r-1\), this yields a contradiction.

The maximal destabilizing subbundle is canonically defined from the unordered set of points \(\{q_1, \dots, q_g\}\), the data of the cover \(\alpha \colon E' \to E\), and the generator \(\tau\). 
Fixing \(\alpha\colon E' \to E\) and \(\tau\), the determinant of \(F\) gives a rational map
\[d \colon \Sym^g E' \to \Pic^{\deg F} E,\]
which necessarily extends to a regular map since \(\Pic^{\deg F} E\) is a torsor under an abelian variety and \(\Sym^g E'\) is smooth. 
Any such map factors through the Abel--Jacobi map \(\Sym^gE' \to \Pic^g E'\).  We thus have a canonically-defined map
\[d \colon \Pic^g E' \to \Pic^{\deg F}E\]
from the data of the cyclic cover \(\alpha \colon E' \to E\) and the generator \(\tau\).  We will use the naturality of this map \(d\) to deduce that \(g\) divides \(\deg F\).

Translation on \(E'\) by a point \(\theta \in \Pic^0E'\), and on \(E\) by the norm \(\alpha_*\theta \in \Pic^0E\), commute with the map \(\alpha\):
\begin{center}
    \begin{tikzcd}
    E' \arrow{r}{+\theta} \arrow{d}{\alpha} & E' \arrow{d}{\alpha}\\
    E \arrow{r}{+\alpha_*\theta} & E
    \end{tikzcd}
\end{center}
Since translation acts as the identity on \(\Pic^0E'\), this automorphism also preserves \(\tau\).  
Specialize to the case that \(\theta\) is a point of order \(g\) for which \(\alpha_*\theta\) also has order \(g\). (Such a point always exists since the kernel of \(\alpha_*\) is isomorphic to \(\zz/r\zz\), while the \(g\)-torsion is isomorphic to \(\zz/g\zz \times \zz/g\zz\).) Translation by \(\theta\) therefore acts as the identity on \(\Pic^gE'\).  Hence it must act by the identity on \(\Pic^{\deg F}E\). 
But translation by \(\alpha_*\theta\) acts on \(\Pic^{\deg F}E\) by addition of \((\deg F)\alpha_*\theta\).  In order for this to be the identity, \(g\) must divide \(\deg F\).
\end{proof}

\section{Admissible covers and the general case}\label{sec-generalcase}

In this section we prove Theorem \ref{thm:tsch}. We use the theory of admissible covers to inductively reduce to the case when \(h=1\).  

Recall that a finite, surjective, generically \'etale cover of nodal curves \(\alpha \colon X \to Y\) is an  \defi{admissible cover} if for every point $x\in X$ whose image is a singular point $y \in Y$ the ramification indices on both branches at $x$ are equal to a common positive integer $n$. In other words, there are isomorphisms of completed local rings  

\begin{center}
    \begin{tikzcd}
    \hat{\O}_{X,x} \arrow{r}{\simeq}  & k[[s,t]]/(st)  \\
    \hat{\O}_{Y,y} \arrow{r}{\simeq} \arrow{u}& k[[u,v]]/(uv)\arrow[swap]{u}{(u,v)\mapsto (s^n,t^n)}
    \end{tikzcd}
\end{center}
Then such a cover may be deformed to a smooth cover. This was proved by Harris and Mumford in \cite{hm82} with some extraneous assumptions which were later removed. We refer the reader to  \cite[Proposition 5.4]{liu} for a nice exposition.

\subsection{Tschirnhausen bundles and admissible covers}

We will consider admissible covers of the following shape.
Let \(Y_1 \cup_p Y_2\) be a nodal curve composed of two smooth curves meeting at a point \(p\).  Let 
\[\alpha \colon X_1 \cup X_2 \to Y_1 \cup Y_2\]
be an admissible cover and write 
\[\alpha|_{X_1} \colon X_1 \to Y_1 \quad \text{and} \quad \alpha|_{X_2} \colon X_2 \to Y_2\]
for the maps on each component.

\begin{lem}
Suppose that \(\alpha \colon X_1 \cup X_2 \to Y_1 \cup_p Y_2\) is an admissible cover that is \'etale over \(p\).  For any line bundle \(L\) on \(X_1 \cup X_2\), the pushforward \(\alpha_*L\) is a vector bundle satisfying \((\alpha_*L)|_{Y_i} \simeq \alpha|_{X_i}{}_*(L|_{X_i})\).  In particular, we have 
\( \DTs(\alpha)|_{Y_i} \simeq \DTs(\alpha_i)\).
\end{lem}
\begin{proof}
Since \(\alpha\) is \'etale over the node, the scheme-theoretic preimage of \(Y_i\) is \(X_i\).
Hence we have a Cartesian square
\begin{center}
\begin{tikzcd}
X_i \arrow{r} \arrow{d} & X_1 \cup X_2 \arrow{d} \\
Y_i \arrow{r} & Y_1 \cup Y_2.
\end{tikzcd}
\end{center}
Since \(\alpha\) is finite, the higher direct images \(R^i \alpha_* L\) vanish. The theorem on cohomology and base change thus provides the desired isomorphism
\((\alpha_*L)|_{Y_i} \to \alpha|_{X_i}{}_*(L|_{X_i})\).
\end{proof}

By \cite[Lemma 4.1]{clv}, a vector bundle on the reducible nodal curve \(Y_1 \cup_p Y_2\) is stable if its restriction to one component is semistable, and its restriction to the other component is stable.  Using this, we obtain the following corollary.

\begin{cor}\label{cor-inductive}
Suppose that \(\alpha \colon X_1 \cup X_2 \to Y_1 \cup_p Y_2\) is an admissible cover that is \'etale over \(p\).  If \(\DTs(\alpha_1)\) is stable and \(\DTs(\alpha_2)\) is semistable, then \(\DTs(\alpha)\) is stable.
\end{cor}

\subsection{Proof of Theorem \ref{thm:tsch}} 
We will reduce part \eqref{main_h1} to  Proposition \ref{prop-coverelliptic}. Let $E$ be an elliptic curve and let $\alpha'\colon X \to E$ be the cover constructed in  Proposition \ref{prop-coverelliptic} so that $\DTs(\alpha')$ is semistable. By \cite[Proposition 5.4]{liu}, $\alpha'$ can be smoothed. Let $\beta$ be a general smoothing of $\alpha'$. By the openness of semistability \cite[Proposition 2.3]{clv}, $\DTs(\beta)$ is semistable. By the Riemann--Hurwitz formula, the map $\beta$ cannot be \'etale, hence by Proposition \ref{prop-unstable} $\beta$ must be primitive. By Proposition \ref{prop-Hurwitz}, the Hurwitz space $\mathscr{H}_{r,g}(E)$ parameterizing primitive, simply-branched, degree $r$, genus $g$ covers of $E$ is irreducible. Hence $\beta$ lies in $\mathscr{H}_{r,g}(E)$ and $\DTs(\beta)$ is semistable for the general primitive cover.

 To prove part \eqref{main_h2} of the theorem, we argue by induction on \(h\). Specialize \(Y\) to the union \(Y' \colonequals Y_1 \cup Y_2\), where \(Y_1\) is an elliptic curve and \(Y_2\) is of genus \(h - 1\), which meet at a node \(p\).

Let \(\alpha_1\colon X_1 \to Y_1\) be a general primitive degree \(r\) cover of \(Y_1\) with two ramification points such that \(\alpha_1\) is unramified over \(p\). Let \(\alpha_2\colon X_2 \to Y_2\) be a general primitive-or-\'etale degree \(r\) cover of \(Y_2\) with two fewer ramification points and unramified over \(p\). By induction, both \(\DTs(\alpha_1)\) and \(\DTs(\alpha_2)\) are semistable.
By Lemma \ref{lem-degTsch}, the degree of \(\DTs(\alpha_1)\) is \(-1\). Consequently, the rank and the degree of \(\DTs(\alpha_1)\) are relatively prime and  the bundle \(\DTs(\alpha_1)\) is in fact stable.

Let \(\alpha'\colon X' \to Y'\) be the cover obtained by gluing  \(X' = X_1 \cup X_2\) so that \(\alpha'|_{X_i} = \alpha_i\).  By \cite[Proposition 5.4]{liu}, we can smooth $\alpha'$. Let $\beta\colon X \to Y$ be a general smoothing of $\alpha'$. By Corollary \ref{cor-inductive},  the Tschirnhausen bundle $\DTs(\beta)$ is stable.  By the Riemann--Hurwitz formula, $\beta$ is not \'etale. Consequently, by Proposition \ref{prop-unstable}, the map $\beta$ must be primitive. Hence,  $\beta$   lies in the irreducible component of primitive covers. We conclude that for the general primitive cover $\DTs(\beta)$ is stable.
\qed

\begin{rem}
When the characteristic of the base field is larger than \(r\), this argument proves something stronger. By Atiyah's classification \cite{at57},  (semi)stable sheaves on a genus one curve are strongly (semi)stable. Observe that the statement of  \cite[Lemma 4.1]{clv} for strong (semi)stability follows formally from the statement of \cite[Lemma 4.1]{clv} for (ordinary) (semi)stability.  We thus conclude the following corollary.

\begin{cor}\label{cor-strong}
Let \(\alpha\colon X \to Y\) be a general primitive degree \(r\) cover, where \(X\) has genus \(g\) and \(Y\) has genus \(h\), over an algebraically closed field of characteristic greater than \(r\).
\begin{enumerate}
    \item If \(h=1\), then \(\Tsch(\alpha)\) is strongly semistable.
    \item If \(h \geq 2\), then \(\Tsch(\alpha)\) is strongly stable.
    \end{enumerate}
\end{cor}
\end{rem}

\subsection{Covers of rational curves}\label{sec-rational} Finally for completeness, we give a proof of Proposition \ref{thm-genus0}.
Let \(\alpha\colon X \to Y \simeq \pp^1\) be a degree \(r\) cover. When \(g=g(X)=0\) or, equivalently,  the degree of the ramification divisor is \(2r-2\), then \(\DTs(\alpha)\) is a vector bundle of rank \(r-1\) and degree \(1-r\). Since \(H^0(\alpha_* \O_X) = H^0(\O_{\pp^1}) \oplus H^0(\DTs(\alpha))\), we conclude that \(H^0(\DTs(\alpha))=0\). Consequently, the $r-1$ degrees occurring in the splitting type must all be negative and sum to $1-r$. Therefore, \(\DTs(\alpha) \simeq \O_{\pp^1}(-1)^{\oplus (r-1)}\) and is balanced. Now we induct on the degree of the ramification divisor. 

Suppose we know that 
when the degree of the ramification divisor is less than \(b\),
the general Tschirnhausen bundle is balanced. For \(b \geq 2r-2\), specialize the cover so that \(Y\) is the union of two rational curves \(Y_1 \cup Y_2\)  meeting at a node. Take a degree \(r\) cover of the reducible \(Y\) so that the degree of the ramification divisor of \(\alpha_1\) is \(2r-2\) and the degree of the ramification divisor of \(\alpha_2\) is \(b - 2r + 2\). Assume that the maps \(\alpha_i\) are not ramified over \(p\).

If \(b - 2r + 2 \geq 2r - 2\), then by induction, \(\DTs(\alpha_2)\) is balanced.
Otherwise, we allow the cover over \(Y_2\) to be disconnected, and take
a general cover of degree \(r' \colonequals (b-2r+2)/2\) together with \(r-r'\) components mapping by the identity to \(Y_2\).
In this case, \(\DTs(\alpha_2) \simeq \O_{\pp^1}^{\oplus(r-r')} \oplus \O_{\pp^1}(-1)^{\oplus(r'-1)}\) is still balanced. Since \(\DTs(\alpha_1)\) is perfectly balanced, we conclude that
a general deformation is  balanced, 
since being balanced is equivalent to \(h^1(\End V) = 0\) and hence is open in families.\qed

\bibliographystyle{plain}

\end{document}